\documentclass[a4paper,11pt]{article}

\usepackage{amssymb}
\usepackage{amsfonts}

\usepackage{amsthm}
\usepackage{times}
\usepackage{cite}
\usepackage{graphicx}
\usepackage{amssymb, amsmath, mathrsfs}
\allowdisplaybreaks
\usepackage{color}

\graphicspath{ {./SpiResults/} }

\newtheorem{thm}{Theorem}

\newtheorem{prop}[thm]{Proposition}
\newtheorem{defn}{Definition}
\newtheorem{rem}{Remark}
\numberwithin{equation}{section}

\definecolor{Green}{rgb}{0.2,0.7,0.3}

\title {A decoupled unconditionally stable numerical scheme for the Cahn-Hilliard-Hele-Shaw system}

\author{
 {\sc Daozhi Han}\footnote{Department of Mathematics, Florida State University, Tallahassee, FL 32306, USA.
Email: \emph{dhan@math.fsu.edu}} 
 }

\date{}

\begin{document}
\maketitle

\begin{abstract}
We propose a novel decoupled unconditionally stable numerical scheme for the simulation of two-phase flow in a Hele-Shaw cell which is governed by the Cahn-Hilliard-Hele-Shaw system (CHHS) with variable viscosity. The temporal discretization of the Cahn-Hilliard equation is based on a convex-splitting of the associated energy functional. Moreover, the capillary forcing term in the Darcy equation is separated from the pressure gradient at the time discrete level by using an operator-splitting strategy. Thus the computation of the nonlinear Cahn-Hilliard equation is completely decoupled from the update of pressure.  Finally, a pressure-stabilization technique is used in the update of pressure so that at each time step one only needs to solve a Poisson equation with constant coefficient.  We show that the scheme is unconditionally stable.  Numerical results are presented to demonstrate  the accuracy and efficiency of our scheme.

\end{abstract}


\section{Introduction}
Consider the Ginzburg-Landau free energy of a binary fluid with matched density (assumed to be $1$)
\begin{align}\label{GLenergy}
E(\phi)=\gamma\int_{\Omega}\frac{1}{\epsilon}F(\phi)+\frac{\epsilon}{2}|\nabla\phi|^2\,dx,
\end{align}
where $\gamma$ is a dimensionless surface tension parameter, $F(\phi)=\frac{1}{4}(\phi^2-1)^2$ is the homogeneous free energy density function,  $\phi$ is the order parameter which takes distinct values $\pm1$ in the respective bulk phase, and $\epsilon$ is a constant measuring the thickness of the transition layer between the two phases. Then the two-phase incompressible flow in a Hele-Shaw cell can be modeled by the following Cahn-Hilliard-Hele-Shaw system \cite{LLG2002a,LLG2002b, Wise2010, FeWi2012, WaZh2013}
\begin{equation}\label{CHD}
\left\{
  \begin{aligned}
  &\partial_t \phi +\nabla \cdot (\phi \mathbf{u})=\frac{1}{Pe}\nabla \cdot \big(m(\phi) \nabla \mu\big),\\
  & \mu=\phi^3-\phi-\epsilon^2\Delta \phi, \\
  &\mathbf{u}=-\frac{1}{12\eta(\phi)}\big(\nabla p+\frac{\gamma}{\epsilon}\phi \nabla \mu \big),\\
  &\nabla \cdot \mathbf{u}=0.
  \end{aligned}
  \right.
\end{equation}
Here $Pe$ is the diffusional Peclet number; $m(\phi)$ and $\eta(\phi)$ are the mobility and kinematic viscosity coefficient, respectively. Throughout, the following assumptions will be assumed 
\begin{align}\label{AssBound}
0< m_1 \leq m(\phi) \leq m_2, \quad 0< \eta_1 \leq \eta(\phi) \leq \eta_2.
\end{align}
One may recognize that the first two equations in the system \eqref{CHD} are the (convective) Cahn-Hilliard equation with $\mu$ the chemical potential, and that the last two equations are the Darcy equation incorporating the elastic forcing term.

We close the system with the following initial and boundary conditions
\begin{align}
&\phi|_{t=0}=\phi_0, \\
&\partial_{\mathbf{n}}\phi|_{\partial \Omega}=0, \label{noflux1}\\
&\partial_{\mathbf{n}}\mu|_{\partial \Omega}=0, \label{contactangel} \\
&\mathbf{u}\cdot \mathbf{n}|_{\partial \Omega}=0.\label{noflux2}
\end{align}
Here $\mathbf{n}$ is the unit outer normal of the boundary $\partial \Omega$;  Eq. \eqref{noflux1} is a Neumann boundary condition for phase field variable which says that the diffuse interface is perpendicular to the physical boundary \cite{Abels2009}; Eq. \eqref{contactangel} means that there is no mass flux through the boundary; Eq. \eqref{noflux2} is the usual no penetration boundary condition for fluid velocity. With boundary conditions \eqref{noflux1}-\eqref{noflux2}, it is clear that the CHHS system \eqref{CHD} is  energy dissipative
\begin{align}\label{continuousenergylaw}
&\frac{dE}{dt} =-\frac{\gamma}{Pe}\int_{\Omega}m(\phi)|\nabla\mu|^2\,dx-\int_{\Omega}12\eta(\phi)|\mathbf{u}|^2\, dx \leq 0.
\end{align}

The CHHS system \eqref{CHD} can be viewed as a simplification of the Cahn-Hilliard-Navier-Stokes system (CHNS) with (nearly) matched density in the Hele-Shaw setting \cite{LLG2002a}. See \cite{LoTr1998, LiSh2003, KKL2004, Feng2006, KaWe2007, ShYa2010c, BoMi2011, HaWa2014} and many others for results related to CHNS system. The applications of the CHHS model and its variant are abundant. In \cite{ShOo1992}, a similar set of equations are employed in the simulation of spinodal decomposition of a binary incompressible fluid in a Hele-Shaw cell. Recently, the CHHS system has been applied in the study of Saffman-Taylor instability \cite{CHM2014} when a more viscous fluid is displaced by a less viscous one resulting in complex pattern formation. Incorporated in a mass source term, the CHHS system also serves as a tumor growth model, cf. \cite{WLC2011}. When accounted for permeability/hydraulic conductivity, the CHHS system, also known as Cahn-Hilliard-Darcy equation (CHD) in the literature, can be used to model multiphase flow in porous media. We refer to \cite{HSW13} and references therein for many potential applications of the CHD system.

We note that the equations in \eqref{CHD} are coupled, highly nonlinear and numerically stiff with large spatial derivative over a small transition layer. Thus solving the CHHS system numerically is challenging. On one hand, unconditionally stable schemes are preferred in order to cope with the stiffness issue. A common strategy in discretizing the nonlinear chemical potential equation (second equation in \eqref{CHD}) in time is based on the convex-splitting of the free energy functional $E$, i.e., treating the convex part of the functional implicitly and concave part explicitly, an idea dates back to Eyre \cite{Eyre1998}. This semi-implicit discretization yields not only unconditional stability but also unconditionally unique solvability \cite{Wise2010, FeWi2012}. On the other hand,  unconditionally stable schemes such as schemes based on convex-splitting tend to be nonlinear and coupled. Recently, efficient nonlinear multigrid solvers have been designed to solve the resulting nonlinear system \cite{Wise2010, FeWi2012, GXX2014}. 

In this work, we propose a novel decoupled unconditionally stable numerical scheme for the CHHS system with variable viscosity. The temporal discretization of the Cahn-Hilliard equation is based on the aforementioned convex-splitting of the energy functional. Moreover, an operator-splitting/fractional-step method is applied to split the computation of pressure gradient and capillary forcing term in the Darcy equation. As a result,  the computation of the nonlinear Cahn-Hilliard equation is completely decoupled from the update of pressure. A similar strategy has been utilized in the computation of a triphasic
              Cahn-Hilliard-Navier-Stokes model with variable density in \cite{Minjeaud2013}.  Finally, a pressure-stabilization technique is used in the update of pressure so that  only a Poisson equation with constant coefficient needs to be solved at each time step.  The scheme is shown to be unconditionally stable.  

The rest of the article is organized as follows. In Section 2, we motivate and introduce the semi-discrete in time numerical scheme. The scheme is further discretized in space by finite element method in Section 3. We show that the fully discrete scheme is unconditionally stable. In section 4, we provide numerical evidence that the scheme is convergent and efficient. We remark that though we solve the nonlinear Cahn-Hilliard equation by Newton's method in our simulation, in principle the scheme can be combined with the nonlinear multigrid solver developed in \cite{FeWi2012}.

\section{A semi-discrete numerical scheme }
Let $N$ be a positive integer and $0=t_0 <t_1< \cdots <t_N=T$ be a uniform partition of $[0,T]$. Denote by $k:=t_n -t_{n-1}$, $n=1,2 \ldots N$,  the time step size.
We propose the following discrete in time, continuous in space numerical scheme for the computation of the system \eqref{CHD} under the boundary conditions \eqref{noflux1}-\eqref{noflux2}: seek $\{\phi^{n+1}, \mu^{n+1}, p^{n+1}\}$ such that

\begin{align}
  &\frac{\phi^{n+1}-\phi^n}{k}+\nabla \cdot (\phi^n \mathbf{u}^{n+1})=\frac{1}{Pe}\nabla \cdot \big(m(\phi^n)\nabla \mu^{n+1}\big) , \label{CHO}\\
  &\mu^{n+1}=(\phi^{n+1})^3- \phi^n-\epsilon^2 \Delta \phi^{n+1}, \label{CHCh}\\
  &\Delta(p^{n+1}-p^n)=12\eta_1 \nabla \cdot \mathbf{u}^{n+1}, \label{Pre}\\
  &\partial_{\mathbf{n}} \phi^{n+1}\big|_{\partial \Omega}=\partial_{\mathbf{n}} \mu^{n+1}\big|_{\partial \Omega}=\partial_{\mathbf{n}} p^{n+1}\big|_{\partial \Omega}=0,  \label{Boun}
\end{align}
where the velocity is given by
\begin{equation}\label{Vel}
  \mathbf{u}^{n+1}=-\frac{1}{12\eta(\phi^n)}(\nabla p^n+\frac{\gamma}{\epsilon}\phi^n \nabla \mu^{n+1}),  
\end{equation}

Several remarks are in order. First, we note that the semi-implicit discretization of the nonlinear term $(\phi^{n+1})^3-\phi^n$ in Eq. \eqref{CHCh} is derived from a convex splitting of the function $F(\phi)$, an idea dates back to Eyre \cite{Eyre1998} (see also \cite{Wise2010}). Owing to the convexity, it is straightforward to verify that the following inequality holds
\begin{align}\label{F:ine}
  F(\phi^{n+1})-F(\phi^n) \leq [(\phi^{n+1})^3-\phi^n](\phi^{n+1}-\phi^n).
\end{align}
Next, note that the pressure in the velocity equation \eqref{Vel} is explicit, thus upon substitution of $\mathbf{u}^{n+1}$ into the Cahn-Hilliard equation \eqref{CHO} allowing for decoupling the computation of Cahn-Hilliard system \eqref{CHO}--\eqref{CHCh} from the pressure equation \eqref{Pre}.     This is in particular contrast to the previous work \cite{Wise2010, GXX2014} where one has to solve a coupled nonlinear system by multigrid method. Another advantage of the scheme  \eqref{Vel} and \eqref{Pre} is that at each time step one needs to solve a pressure Poisson equation with only constant coefficient for which many efficient solvers exist.

The time discretization of the Darcy equation can be motivated from the standpoint of operator-splitting (fractional step method). There are two contributing forces in the Darcy equation: the pressure gradient and capillary force. The velocity $\mathbf{u}^{n+1}$ in Eq. \eqref{Vel} can be viewed as an intermediate velocity that takes into account the capillary forcing term. The pressure gradient from previous time step is included in $\mathbf{u}^{n+1}$ for accuracy, a modification similar to the incremental pressure projection method for Navier-Stokes equation (cf. \cite{GMS2006}). Then the true velocity must correct the intermediate velocity and satisfy
\begin{align} \label{TVel}
\tilde{\mathbf{u}}^{n+1}-\mathbf{u}^{n+1}=-\frac{1}{12\eta(\phi^n)}(\nabla p^{n+1} - \nabla p^n), \quad \nabla \cdot \tilde{\mathbf{u}}^{n+1}=0.
\end{align}
It is clear that adding Eqs. \eqref{Vel} and \eqref{TVel} together will recover the Darcy equation in \eqref{CHD}. The true velocity $\tilde{\mathbf{u}}^{n+1}$ can be eliminated once one applies the divergence operator to the first equation in \eqref{TVel}
\begin{align}\label{DivApp}
\nabla \cdot [\frac{1}{12\eta(\phi^n)}(\nabla p^{n+1} - \nabla p^n)] =\nabla \cdot \mathbf{u}^{n+1}.
\end{align}
As is in the case of Navier-Stokes equation with variable density \cite{GuSa2009}, we note that Eqs. \eqref{Vel} and \eqref{DivApp} can be interpreted as a direct approximation of the following perturbed Darcy equation
\begin{align}
  &\mathbf{u}=-\frac{1}{12\eta(\phi)}\big(\nabla p+\frac{\gamma}{\epsilon}\phi \nabla \mu \big), \label{PDarcy1}\\
& \nabla \cdot \mathbf{u}- k\nabla \cdot (\frac{1}{12\eta(\phi)}\nabla p_t)=0. \label{PDarcy2}
\end{align}
Eqs.\eqref{PDarcy1}--\eqref{PDarcy2} can be viewed as a penalty method/pseudo-compressibility method for the Darcy equation with variable viscosity. Formally, Eq. \eqref{PDarcy2} is a $\mathcal{O} (k)$ approximation of the divergence free condition irrespective of the value of $\eta(\phi)$. The choice of a constant $\eta_1$ in Eq. \eqref{PDarcy2} leads to the discretization Eq. \eqref{Pre}, and is crucial in proving the unconditional stability of the scheme, see Section 3 below. We remark that the pseudo-compressibility technique, also known as pressure-stabilization,  is well-known in the computation of incompressible flow with constant density, cf. \cite{Rannacher1992, Shen1996b} and references therein. Recently, it has been generalized to solving incompressible Navier-Stokes equation with variable density \cite{GuSa2009}.

\section{Fully discrete formulation}
\subsection{Finite element formulation}
We note that Eqs. \eqref{CHD} can be reformulated without using the explicit velocity.
\begin{align}
&\partial_t \phi +\nabla \cdot \big[\frac{\phi}{12\eta(\phi)}(\nabla p+\frac{\gamma}{\epsilon} \phi\nabla \mu)\big]=\frac{1}{Pe}\nabla \cdot \big(m(\phi) \nabla \mu\big), \label{RCH}\\
&\mu=\phi^3-\phi-\epsilon^2\Delta \phi, \label{RChePe}\\
&\nabla \cdot \big[\frac{1}{12\eta(\phi)}(\nabla p+\frac{\gamma}{\epsilon} \phi \nabla \mu)\big]=0. \label{RPre}
\end{align}
The corresponding boundary conditions become
\begin{align}
\nabla \phi \cdot \mathbf{n}|_{\partial \Omega}=0, \quad \nabla \mu \cdot \mathbf{n}|_{\partial \Omega}=0, \quad \nabla p \cdot \mathbf{n}|_{\partial \Omega}=0. \label{RBoun}
\end{align}

A weak formulation and solutions to the initial-boundary value problem \eqref{RCH}--\eqref{RBoun} can be defined similarly as \cite{FeWi2012}. Here we consider mainly the 2D case. $H^1(\Omega)$ is the usual Hilbert space, and $L^2_0(\Omega)$ is a subspace of $L^2(\Omega)$ whose elements have zero average.
\begin{defn}
Let $\phi_0 \in H^1(\Omega)$. A triple $\{\phi, \mu, p\}$ is called a weak solution of problem \eqref{RCH}-\eqref{RBoun} if it satisfies 
\begin{align}
& \phi \in L^\infty(0, T; H^1(\Omega)) \cap L^4(0,T; L^\infty(\Omega)),  \\
&\partial_t \phi \in L^{\frac{8}{5}}(0,T; (H^1(\Omega))^\prime), \\
&\mu \in L^2(0,T; H^1(\Omega)), \\
& \frac{1}{12\eta(\phi)}(\nabla p + \frac{\gamma}{\epsilon} \phi \nabla \mu) \in L^2(0,T; \mathbf{L}^2(\Omega)), \\
&p \in  L^{\frac{8}{5}}(0,T; H^1(\Omega)\cap L^2_0(\Omega)),
\end{align}
and there hold, for almost all $t \in (0,T)$
\begin{align}
&\langle \partial_t \phi, v\rangle+\big(\frac{\phi}{12\eta(\phi)}[\nabla p+\frac{\gamma}{\epsilon} \phi \nabla \mu], \nabla v\big)+\frac{1}{Pe}(m(\phi)\nabla \mu, \nabla v)=0, \quad \forall v \in H^1(\Omega), \label{WCH}\\
&(\mu, \varphi)-(\phi^3-\phi, \varphi)-\epsilon^2 (\nabla \phi, \nabla \varphi)=0, \quad \forall \varphi \in H^1(\Omega),  \label{WChePe} \\
&\big(\frac{1}{12\eta(\phi)}[\nabla p+\frac{\gamma}{\epsilon} \phi \nabla \mu], \nabla q \big)=0, \quad \forall q \in H^1(\Omega),  \label{WPre}
\end{align}
with initial condition $\phi(0)=\phi_0$.
\end{defn}

Under the boundedness assumption \eqref{AssBound} on $  \eta(\phi), m(\phi) $, the existence of such a weak solution can be established similarly as \cite{FeWi2012} (see also \cite{WaZh2013, HWW2014}).

Let $\mathcal{T}_h$ be a quasi-uniform triangulation of the domain $\Omega$ of mesh size $h$.  We introduce $Y_h$ the finite element approximation of  $H^1(\Omega)$  based on the triangulation $\mathcal{T}_h$. In addition, we define $M_h=Y_h \cap L^2_0(\Omega):= \{q_h \in Y_h; \int_{\Omega}q_h dx=0\}$. We assume that $Y_h \times Y_h$ is a stable pair for the biharmonic operator in the sense that there holds the inf-sup condition
\begin{align*}
\sup_{\phi_h \in Y_h} \frac{(\nabla \phi_h, \nabla \varphi_h)}{||\phi_h||_{H^1}} \geq c||\varphi_h||_{H^1}, \quad \forall \varphi_h \in Y_h.
\end{align*}

We now introduce the fully discrete finite element formulation for problem \eqref{CHD} based on the time discretization \eqref{CHO}-\eqref{Vel} and the weak formulation \eqref{WCH}-\eqref{WChePe}: find $\{\phi^{n+1}_h, \mu^{n+1}_h, p^{n+1}_h\} \in Y_h \times Y_h \times M_h$ such that
\begin{align}
&\big( \frac{\phi_h^{n+1}-\phi_h^n}{k}, v_h\big)+\big(\frac{\phi_h^n}{12\eta(\phi_h^n)}[\nabla p_h^n+\frac{\gamma}{\epsilon} \phi^n_h \nabla \mu_h^{n+1}], \nabla v_h\big) \nonumber \\
&\quad \quad +\frac{1}{Pe}(m(\phi_h^n)\nabla \mu_h^{n+1}, \nabla v_h)=0, \quad \forall v_h \in Y_h,  \label{FCH}\\
&(\mu_h^{n+1}, \varphi_h)-\big((\phi_h^{n+1})^3-\phi_h^n, \varphi_h\big)-\epsilon^2 (\nabla \phi_h^{n+1}, \nabla \varphi_h)=0, \quad \forall \varphi_h \in Y_h,  \label{FChePe} \\
&\big(\nabla (p_h^{n+1}-p_h^n), \nabla q_h\big)=-\big(\frac{\eta_1}{\eta(\phi_h^n)}[\nabla p_h^n+\frac{\gamma}{\epsilon} \phi_h^n \nabla \mu_h^{n+1}], \nabla q_h \big), \quad \forall q_h \in Y_h,  \label{FPre}
\end{align}
with initial condition $\phi_h^0=\phi_{0h}$, where $\phi_{0h}$ is the projection of $\phi_0$ in $Y_h$.

\begin{rem}
In the case of constant viscosity coefficient, i.e. $\eta(\phi)\equiv \eta_1$, Eq. \eqref{FPre} reduces to 
\begin{align*}
\big(\nabla (p_h^{n+1}+\frac{\gamma}{\epsilon} \phi_h^n \nabla \mu_h^{n+1}), \nabla q_h\big)=0, \quad \forall q_h \in Y_h, 
\end{align*} 
which is the finite element counterpart of Eq. \eqref{WPre}. The authors in \cite{Wise2010, FeWi2012, GXX2014} treat exclusively the case of constant viscosity. It is remarkable that the scheme with explicit pressure in Eq. \eqref{FCH} is still unconditionally stable.
\end{rem}

\subsection{Stability of the fully discrete scheme}
Our aim in this subsection is to show that the fully discrete scheme \eqref{FCH}--\eqref{FPre} is energy stable for all $h, k, \epsilon >0$. Without ambiguity, we denote by $(f,g)$ the $L^2$ inner product between functions $f$ and $g$. First of all, we claim that at each time step Eq. \eqref{FCH}--\eqref{FPre} are uniquely solvable.

\begin{prop}
For any mesh parameters $k, h$ and any $\epsilon >0$, there exists a unique solution $\{\phi_h^{n+1}, \mu_h^{n+1}, p_h^{n+1}\}$ to the scheme \eqref{FCH}--\eqref{FPre}.
\end{prop}
We note that Eq. \eqref{FPre} is decoupled from the Cahn-Hilliard Eqs. \eqref{FCH}--\eqref{FChePe}. The unique solvability of \eqref{FCH}--\eqref{FChePe} can be established by reformulating the equations as a convex minimization problem \cite{KSW2008}. Another approach is to explore the monotonicity  associated with the convex splitting scheme, cf. \cite{HaWa2014}. Here we omit the details for brevity.

Now, we show that the fully discrete scheme is unconditionally stable. For that, we introduce a discrete energy functional 
\begin{align}\label{TDEnF}
E(\phi_h^n)=\gamma\int_{\Omega}\frac{1}{\epsilon}F(\phi_h^n)+\frac{\epsilon}{2}|\nabla\phi_h^n|^2\,dx.
\end{align}

\begin{thm}\label{ThmSta}
Let $\{\phi_h^{n+1}, \mu_h^{n+1}, p_h^{n+1}\}$ be the unique solution of the scheme \eqref{FCH}--\eqref{FPre}. Define $\mathbf{u}_h^{n+1}:=-\frac{1}{12\eta(\phi_h^n)}[\nabla p_h^n+\frac{\gamma}{\epsilon} \phi_h^n \nabla \mu_h^{n+1}]$. Then 
for any $k, h, \epsilon >0$, the scheme \eqref{FCH}--\eqref{FPre} satisfies a modified energy law
\begin{align}\label{FEnL}
&\big(E(\phi_h^{n+1})+\frac{k}{24\eta_1}||\nabla p_h^{n+1}||_{L^2}^2\big)-\big(E(\phi_h^{n})+\frac{k}{24\eta_1}||\nabla p_h^{n}||_{L^2}^2\big) \leq  -6k||\sqrt{\eta(\phi_h^n)}\mathbf{u}_h^{n+1}||^2_{L^2} \nonumber \\
& -\frac{k\gamma}{\epsilon Pe}||\sqrt{m(\phi_h^n)}\nabla \mu_h^{n+1}||^2_{L^2} -\frac{\gamma \epsilon}{2}||\nabla(\phi^{n+1}_h-\phi^n_h)||^2_{L^2}.
\end{align}
\end{thm}

\begin{proof}
Utilizing the definition 
\begin{align}\label{DVel}
\mathbf{u}_h^{n+1}:=-\frac{1}{12\eta(\phi_h^n)}[\nabla p_h^n+\frac{\gamma}{\epsilon} \phi_h^n \nabla \mu_h^{n+1}],
\end{align}
 one sees that the scheme \eqref{FCH}--\eqref{FPre} can be reformulated as
 \begin{align}
&\big( \frac{\phi_h^{n+1}-\phi_h^n}{k}, v_h\big)-\big(\phi_h^n\mathbf{u}_h^{n+1}, \nabla v_h\big) +\frac{1}{Pe}(m(\phi_h^n)\nabla \mu_h^{n+1}, \nabla v_h)=0, \quad \forall v_h \in Y_h,  \label{RFCH}\\ 
&(\mu_h^{n+1}, \varphi_h)-\big((\phi_h^{n+1})^3-\phi_h^n, \varphi_h\big)-\epsilon^2 (\nabla \phi_h^{n+1}, \nabla \varphi_h)=0, \quad \forall \varphi_h \in Y_h,  \label{DFChePe} \\
&\big(\nabla (p_h^{n+1}-p_h^n), \nabla q_h\big)=12\eta_1\big(\mathbf{u}_h^{n+1}, \nabla q_h \big), \quad \forall q_h \in Y_h. \label{DFPre}
\end{align}

Taking the test function $v_h=k \mu_h^{n+1}$ in Eq. \eqref{RFCH} gives
\begin{align}
\big(\phi_h^{n+1}-\phi_h^n, \mu_h^{n+1}\big)-k\big(\phi_h^n\mathbf{u}_h^{n+1}, \nabla \mu_h^{n+1}\big)+\frac{k}{Pe}||\sqrt{m(\phi_h^n)}\nabla \mu_h^{n+1}||^2_{L^2}=0. \label{EOCH}
\end{align}
Next, we test Eq. \eqref{DFChePe} with $\varphi_h= -(\phi_h^{n+1}-\phi_h^n)$. By utilizing  the identity $2a(a-b)=a^2-b^2+(a-b)^2$, one obtains
\begin{align}
&-\big(\mu_h^{n+1}, \phi_h^{n+1}-\phi_h^n\big)+\big((\phi_h^{n+1})^3-\phi_h^n, \phi_h^{n+1}-\phi_h^n\big) \nonumber \\
&+\frac{\epsilon^2}{2}[||\nabla \phi_h^{n+1}||_{L^2}^2-||\nabla \phi_h^{n}||_{L^2}^2+||\nabla (\phi_h^{n+1}-\phi_h^n)||_{L^2}^2]=0. \label{EChePo}
\end{align}

Adding  Eq. \eqref{EOCH} and inequality \eqref{EChePo} together, in view of the inequality \eqref{F:ine}, one has
\begin{align*}
&\big(F(\phi_h^{n+1})-F(\phi_h^n), 1\big) +\frac{\epsilon^2}{2}[||\nabla \phi_h^{n+1}||_{L^2}^2-||\nabla \phi_h^{n}||_{L^2}^2]-k\big(\phi_h^n\mathbf{u}_h^{n+1}, \nabla \mu_h^{n+1}\big) \nonumber \\ 
&\leq - \frac{k}{Pe}||\sqrt{m(\phi_h^n)}\nabla \mu_h^{n+1}||^2_{L^2}-\frac{\epsilon^2}{2}||\nabla (\phi_h^{n+1}-\phi_h^n)||_{L^2}^2. 
\end{align*}
The preceding inequality, upon multiplied by $\frac{\gamma}{\epsilon}$,   can be written as
\begin{align}
&E(\phi_h^{n+1})-E(\phi_h^n)-\frac{k\gamma}{\epsilon}\big(\phi_h^n\mathbf{u}_h^{n+1}, \nabla \mu_h^{n+1}\big) \leq -\frac{k \gamma}{\epsilon Pe}||\sqrt{m(\phi_h^n)}\nabla \mu_h^{n+1}||^2_{L^2} \nonumber \\ 
&-\frac{\gamma\epsilon}{2}||\nabla (\phi_h^{n+1}-\phi_h^n)||_{L^2}^2. \label{EOChe}
\end{align}

Now we take inner product of Eq. \eqref{DVel} with $12k\eta(\phi_h^n)\mathbf{u}_h^{n+1}$ to get
\begin{align}
12k||\sqrt{\eta(\phi_h^n)}\mathbf{u}_h^{n+1}||_{L^2}^2=-k\big(\nabla p_h^n, \mathbf{u}_h^{n+1}\big)-\frac{k\gamma}{\epsilon}\big(\phi_h^n\nabla \mu_h^{n+1}, \mathbf{u}_h^{n+1}\big). \label{EVel}
\end{align}
We proceed to take the test function $q_h=\frac{k}{12\eta_1}p_h^n$ in Eq. \eqref{DFPre}. We have
\begin{align}
\frac{k}{24\eta_1}[||\nabla p_h^{n+1}||^2_{L^2}-||\nabla p_h^{n}||^2_{L^2}]=\frac{k}{24\eta_1}||\nabla(p_h^{n+1}-p_h^n)||_{L^2}^2 + k\big(\nabla p_h^n, \mathbf{u}_h^{n+1}\big). \label{EPre}
\end{align}
To control $\frac{k}{24\eta_1}||\nabla(p_h^{n+1}-p_h^n)||_{L^2}^2$, we test Eq. \eqref{DFPre} with $q_h=p_h^{n+1}-p_h^n$ and apply the Cauchy-Schwartz inequality so that
\begin{align}
||\nabla (p_h^{n+1}-p_h^n)||_{L^2}^2 \leq 12 \eta_1||\mathbf{u}_h^{n+1}||_{L^2}||\nabla (p_h^{n+1}-p_h^n)||_{L^2}. \nonumber
\end{align}
It follows that 
\begin{align}
\frac{k}{24\eta_1}||\nabla(p_h^{n+1}-p_h^n)||_{L^2}^2 &\leq 6k \eta_1 ||\mathbf{u}_h^{n+1}||^2_{L^2} \leq 6k||\sqrt{\eta(\phi_h^n)}\mathbf{u}_h^{n+1}||_{L^2}^2, \label{Pre:ine}
\end{align}
where the last inequality follows from the assumption $\eta_1 \leq \eta(\phi)$.
Taking sum of Eq. \eqref{EVel} and Eq. \eqref{EPre}, using the inequality \eqref{Pre:ine}, one concludes that
\begin{align}
\frac{k}{24\eta_1}[||\nabla p_h^{n+1}||^2_{L^2}-||\nabla p_h^{n}||^2_{L^2}] \leq-6k||\sqrt{\eta(\phi_h^n)}\mathbf{u}_h^{n+1}||_{L^2}^2 -\frac{k\gamma}{\epsilon}\big(\phi_h^n\nabla \mu_h^{n+1}, \mathbf{u}_h^{n+1}\big). \label{EVelPre}
\end{align}

The modified energy law \eqref{FEnL} then follows from the sum of inequality \eqref{EOChe} and inequality \eqref{EVelPre}.
\end{proof}

\section{Numerical Experiments}
In this section, we perform some numerical tests to verify the accuracy and efficiency of the numerical scheme \eqref{RCH}--\eqref{RPre}. Throughout, we take $Y_h$ to be the $P1$ or P2 finite element function space. It is known \cite{Ciarlet2002} that such $Y_h \times Y_h$ pair is stable for the approximation of biharmonic operator. In principle, any inf-sup compatible approximation spaces for biharmonic operators can be used. We solve the nonlinear equations \eqref{RCH}-\eqref{RChePe} by the classical Newton's method.

\subsection{Convergence, energy dissipation and mass conservation} 
Our aim here is to show numerically that our scheme is first order accurate in time, energy-dissipative and mass-conservative. We consider the problem in a unit square $\Omega=[0,1]\times [0,1]$ with the following initial condition for $\phi$
\begin{align}\label{ConvIni}
\phi_0=0.24\cos(2\pi x)\cos(2\pi y)+0.4\cos(\pi x)\cos(3\pi y).
\end{align}
We impose homogeneous Neumann boundary condition for both $\phi$ and $\mu$, and no-flow boundary condition for velocity (hence homogeneous Neumann boundary condition for pressure).

As the Cahn-Hilliard equation does not have a natural forcing term which can be employed to manufacture exact solutions, we verify the convergence rate by Cauchy convergence test. Specifically, we discretize the domain with a uniform triangulation of spacing $h=\frac{\sqrt{2}}{2^n}$, for $n=5, 6, \cdots 9$, i.e., $2^n+1$ grid points in $x$ and $y$ directions, respectively.
The final time is $T=0.2$. We calculate the rate at which the Cauchy difference of the computed solutions at successive resolution converges to zero in the $H^1$ and $L^2$ norm, respectively. Given $\eta_1\leq \eta_2$, we take a truncated viscosity function as follows
\begin{equation}\label{TrunVisFun}
\eta(\phi)=\left\{
\begin{aligned}
& \eta_1,  \text{ for } \phi >1, \\
& \frac{1+\phi}{2} \eta_1 + \frac{1-\phi}{2} \eta_2, \text{ for } \phi \in [-1, 1], \\
& \eta_2,  \text{ for } \phi <-1. \\
\end{aligned}
\right.
\end{equation}
Thus the truncated viscosity function satisfies $\eta_1 \leq \eta \leq \eta_2$. We also choose a regularized degenerate mobility function 
\begin{align}\label{DegeMob}
m(\phi)=\sqrt{(1+\phi)^2(1-\phi)^2+\epsilon^2}.
\end{align}
The parameters are $\epsilon=0.05$, $Pe=20$, $\gamma=0.005$, $\eta_1=0.0042$, $\eta_2=0.083$.  We expect that the global error in $\phi$ and $p$ at final time $T$ is $e=\mathcal{O}(k)+\mathcal{O}(h)$ in $H^1$ norm and  $e=\mathcal{O}(k)+\mathcal{O}(h^2)$ in $L^2$ norm. Thus if we choose a linear refinement path as $k=\frac{0.2}{\sqrt{2}}h$, we expect to see first order convergence rate in time.
The results in Table \ref{ConvH1} and \ref{ConvL2} confirm this global first order convergence.
\begin{table}[h!]
 \begin{center}
 \caption{$H^1$ Cauchy convergence test.  The triangulation is uniform in space $h=\frac{\sqrt{2}}{2^n}$, for $n=5, 6, \cdots 9$. The final time is $T=0.2$, and the refinement path is linear $k=\frac{0.2}{\sqrt{2}}h$. The viscosity function and mobility function are defined in \eqref{TrunVisFun} and \eqref{DegeMob}, respectively. The other parameters are  $\epsilon=0.05$, $Pe=20$, $\gamma=0.005$, $\eta_1=0.0042$, $\eta_2=0.083$. The expected Cauchy difference at $T$ measured in $H^1$ norm is $\mathcal{O}(k) + \mathcal{O}(h)=\mathcal{O}(k)$.}
\begin{tabular}{c c c c c c c c}
 \hline 
  & $32-64$ & rate & $64-128$ & rate • &$128-256$  & rate  & $256-512$ \\ 
 \hline 
 $\phi$• & $7.88e-2$ & $1.03$ & $3.85e-2$ & $1.04$ & $1.88e-2$ & $1.04$& $9.16e-3$ \\ 

 $p$ & $7.60e-3$• &$0.68$  &$4.73e-3$  &$  1.00$  & $2.38e-3$ &$ 1.01$ & $1.18e-3$ \\ 
 \hline 
 \end{tabular}  
 
 \label{ConvH1}
  \end{center}
  \end{table}
 \begin{table}[h!]
 \begin{center}
 \caption{$L^2$ Cauchy convergence test.  The setup and parameters are the same as in Table \ref{ConvH1} The  Cauchy difference at $T$ measured in $L^2$ norm is expected to be $\mathcal{O}(k) + \mathcal{O}(h^2)=\mathcal{O}(k)$.}
\begin{tabular}{c c c c c c c c}
 \hline 
  & $32-64$ & rate & $64-128$ & rate • &$128-256$  & rate  & $256-512$ \\ 
 \hline 
 $\phi$• & $5.39e-3$ & $1.06$ & $2.58e-3$ & $1.01$ & $1.28e-3$ & $1.02$& $6.35e-4$ \\ 

 $p$ & $3.78e-4$• &$0.65$  &$2.41e-4$  &$  0.94$  & $1.26e-4$ &$ 1.01$ & $6.27e-5$ \\ 
 \hline 
 \end{tabular}  
 
 \label{ConvL2}
  \end{center}
  \end{table}
  
Next, we verify that our numerical scheme is energy-dissipative. The fully discrete counterpart of the energy functional \eqref{GLenergy} associated with the system \eqref{CHD} is defined as
\begin{align}\label{FDEnFun}
E(\phi_h^{n+1})=\gamma\int_{\Omega}\frac{1}{\epsilon}F(\phi_h^{n+1})+\frac{\epsilon}{2}|\nabla\phi_h^{n+1}|^2\,dx.
\end{align}
On the other hand, one can also define an approximate energy functional associated with the fully discrete numerical scheme \eqref{FCH}-\eqref{FPre} according to the modified energy law \eqref{FEnL}
\begin{align}
E_{app}(\phi_h^{n+1}, p_h^{n+1})=E(\phi_h^{n+1})+\frac{k}{24\eta_1}||\nabla p_h^{n+1}||_{L^2}^2.
\end{align}
One can see that formally $E_{app}(\phi_h^{n+1}, p_h^{n+1})$ is a first order approximation of $E(\phi_h^{n+1})$ for fixed $\eta_1
$. We observe from Fig. \ref{FigEnDe} that both energy functional are non-increasing at each time step where the same parameters as above are used and we have set $h=\frac{\sqrt{2}}{128}$ and $k=0.1$.

\begin{figure}[h!]
\centering
 \includegraphics[width=0.8\linewidth]{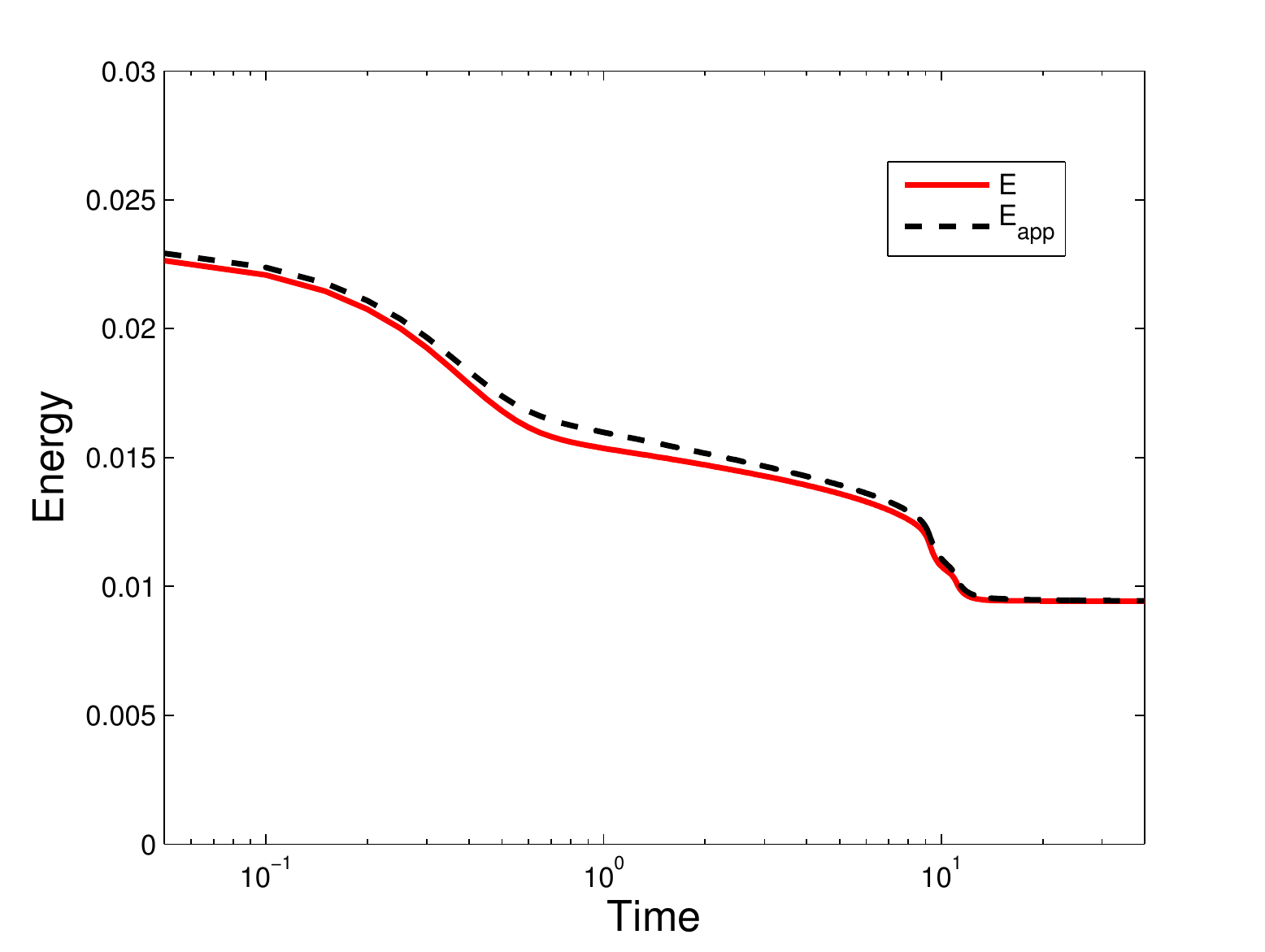}
 \caption{The discrete energy plotted as a function of time for the simulation with initial condition \eqref{ConvIni}, upper dash curve for $E_{app}(\phi_h^{n+1}, p_h^{n+1})$, lower solid curve for $E(\phi_h^{n+1})$. We have set $h=\frac{\sqrt{2}}{128}$ and $k=0.1$. The rest of the parameters are the same as Table \ref{ConvH1}. }
  \label{FigEnDe}
\end{figure}

Finally, we show that our scheme conserves mass, i.e., $\int_{\Omega} \phi_h^n =const.$ for any $n$ such that $nk \leq T$.  Note that $\int_{\Omega}\phi_0 dx=0$. After projection into the P1 finite element space in our computation, we have $\int_{\Omega} \phi_h^0dx =8.14e-6$.  Fig. \ref{Mass} shows that this exact value is preserved during the evolution, which verifies that our scheme is conservative.

\begin{figure}[h!]
\centering
 \includegraphics[width=0.8\linewidth]{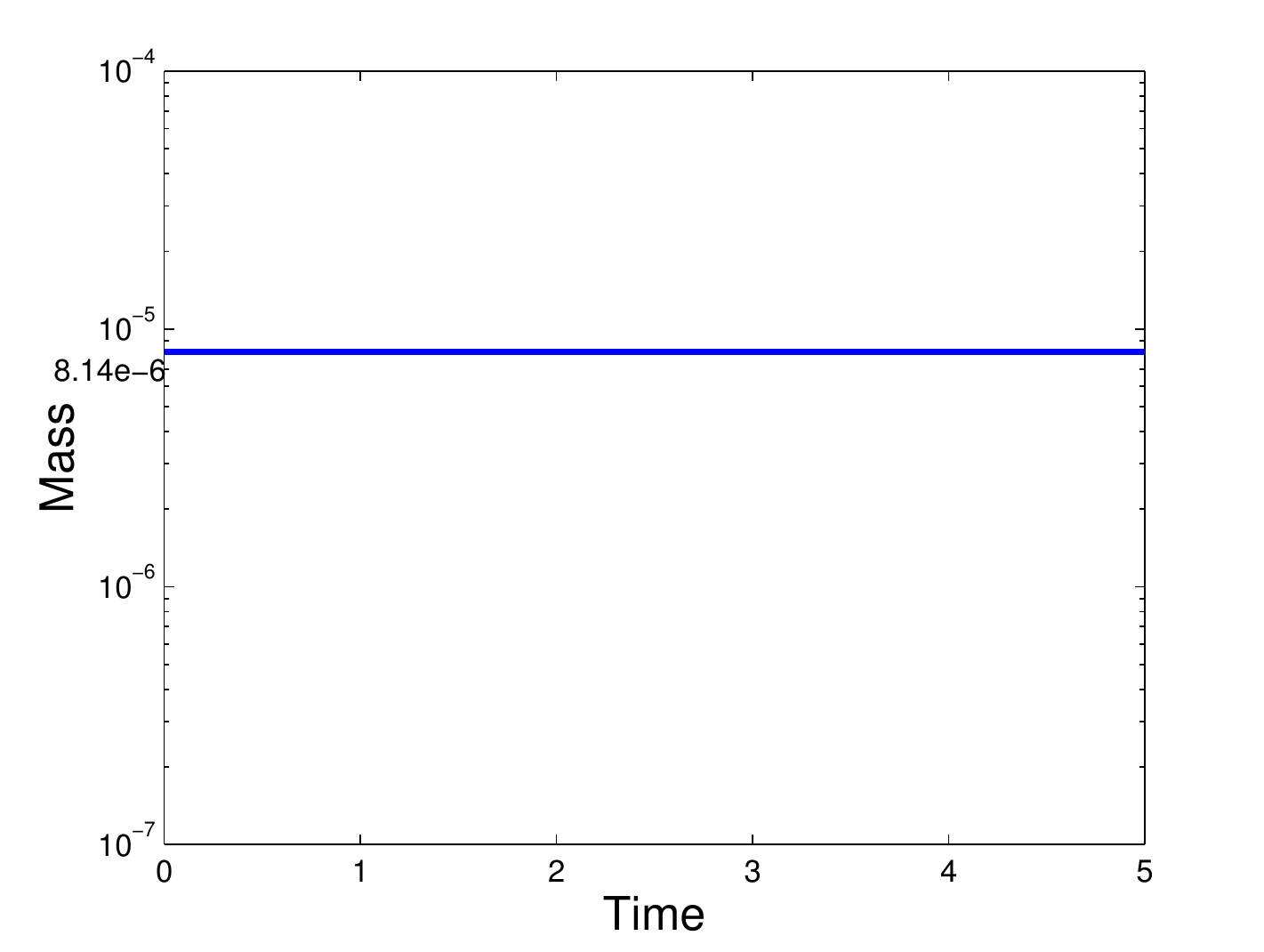}
 \caption{Time evolution of the discrete mass $\int_\Omega \phi_h^n dx$. $h=\frac{\sqrt{2}}{128}$, $k=\frac{1.0}{128}$, and the other parameters are given in Table \ref{ConvH1}.}
  \label{Mass}
\end{figure}

\subsection{Spinodal Decomposition}
To further validate our numerical scheme, we simulate the spinodal decomposition of a binary fluid in a Hele-Shaw cell and examine the effect of $\gamma$ on the coarsening process, see \cite{Wise2010, GXX2014}. Recall that $\gamma$ has the meaning of a scaled surface tension, $\tau=\frac{2\sqrt{2}}{3}\gamma$, where $\tau$ is the physical surface tension \cite{LoTr1998, LLG2002a}. There are two mechanisms responsible for the coarsening process in the Cahn-Hilliard-Hele-Shaw system \eqref{CHD}: chemical diffusion and viscous Darcy dissipation, cf. the energy dissipation inequality \eqref{continuousenergylaw}. 
When $\gamma=0$, the system \eqref{CHD} reduces to the Cahn-Hilliard equation with the surface energy
\begin{align}\label{SGLenergy}
E_{free}(\phi)=\int_{\Omega}\frac{1}{4}(1-\phi^2)^2+\frac{\epsilon^2}{2}|\nabla\phi|^2\,dx,
\end{align}
in which the coarsening process is mediated only by diffusion (no fluid flow). For $\gamma >0$, the flow is surface tension driven.  Larger $\gamma$ would improve the fluid flow, and therefore enhance the viscous Darcy dissipation. Thus the scaled surface energy $E_{free}$ would be smaller when compared at a given time. Note that it is not entirely clear from \eqref{continuousenergylaw} that the scaled surface energy would decay faster for larger $\gamma$, as the variables $\phi, \mu, \mathbf{u}$ depend on $\gamma$ in a nonlinear fashion.

\begin{figure}[h!]
\centering
\begin{tabular}{ccc}

 \includegraphics[scale=0.14]{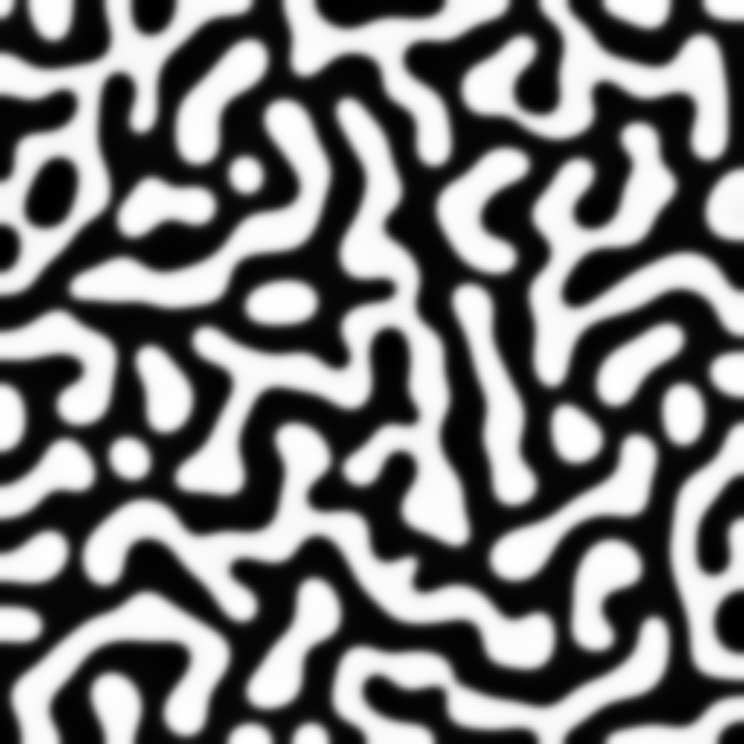}& \includegraphics[scale=0.14]{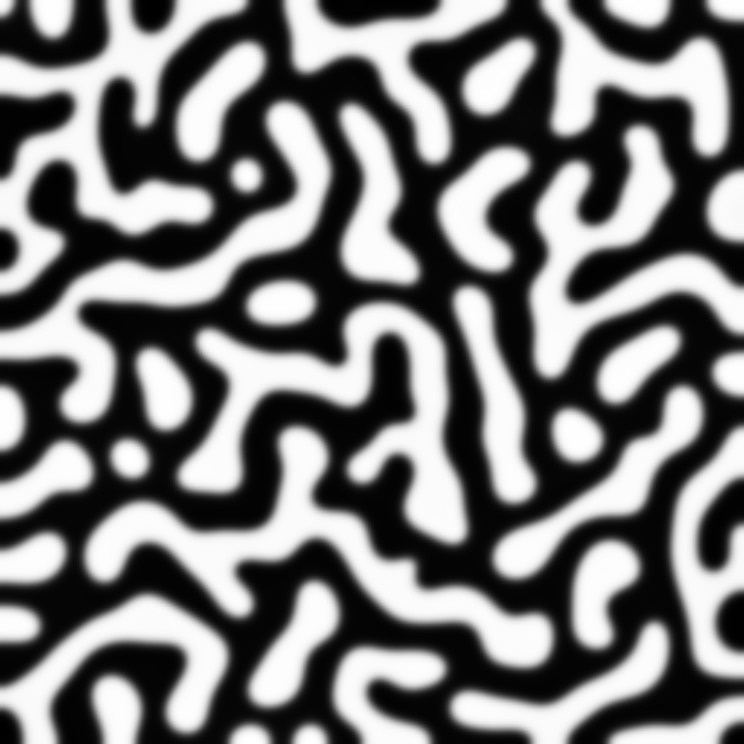} &\includegraphics[scale=0.14]{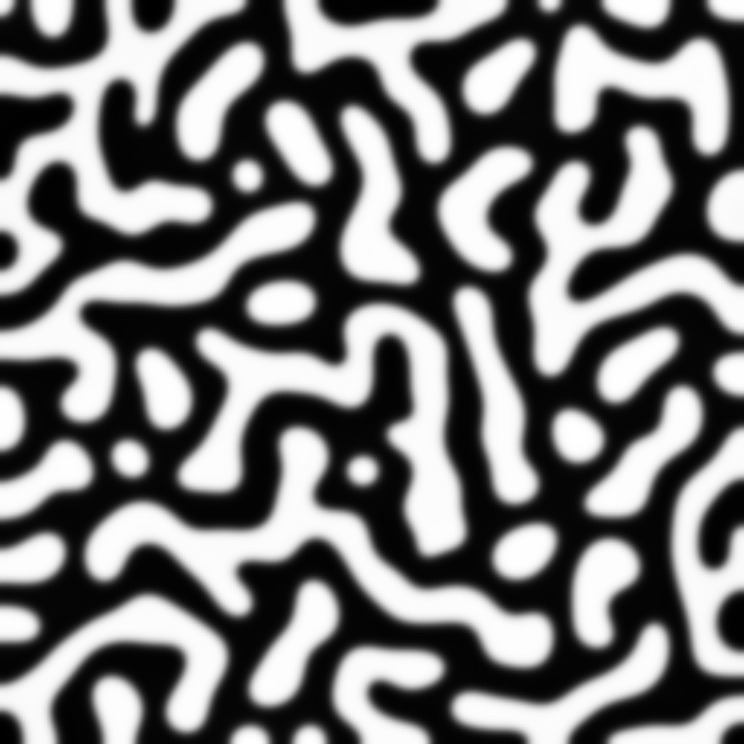} \\

\multicolumn{3}{c}{t=1} \\
 
 \includegraphics[scale=0.14]{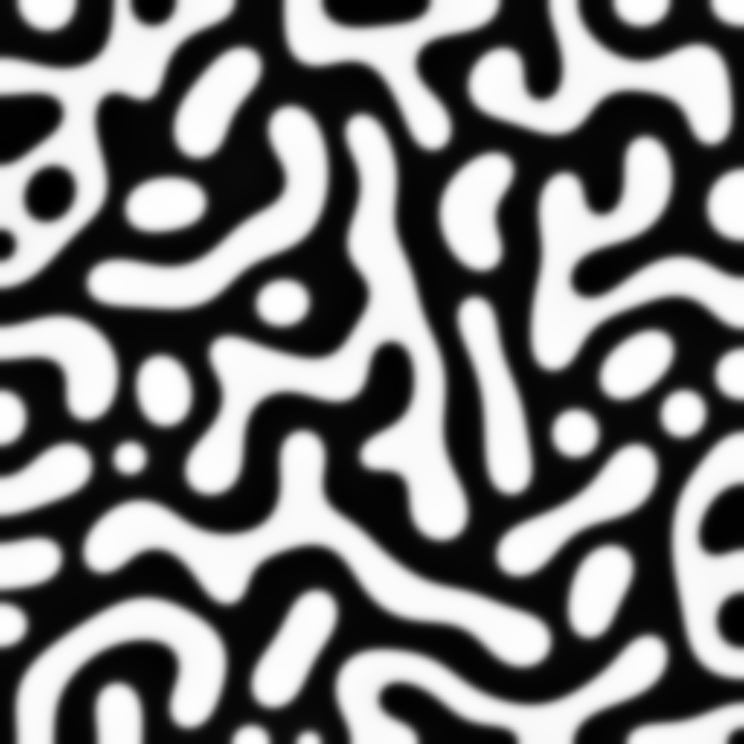}& \includegraphics[scale=0.14]{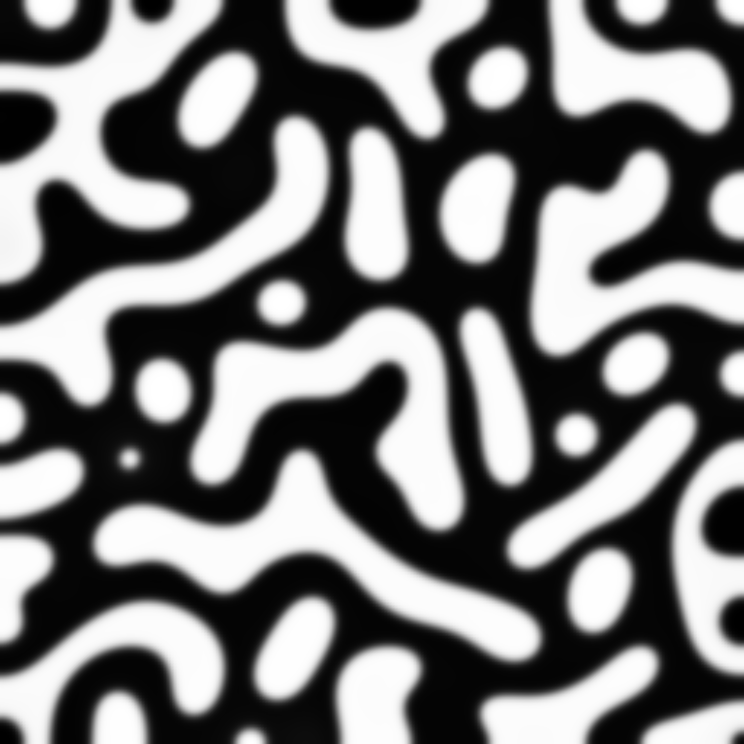} &\includegraphics[scale=0.14]{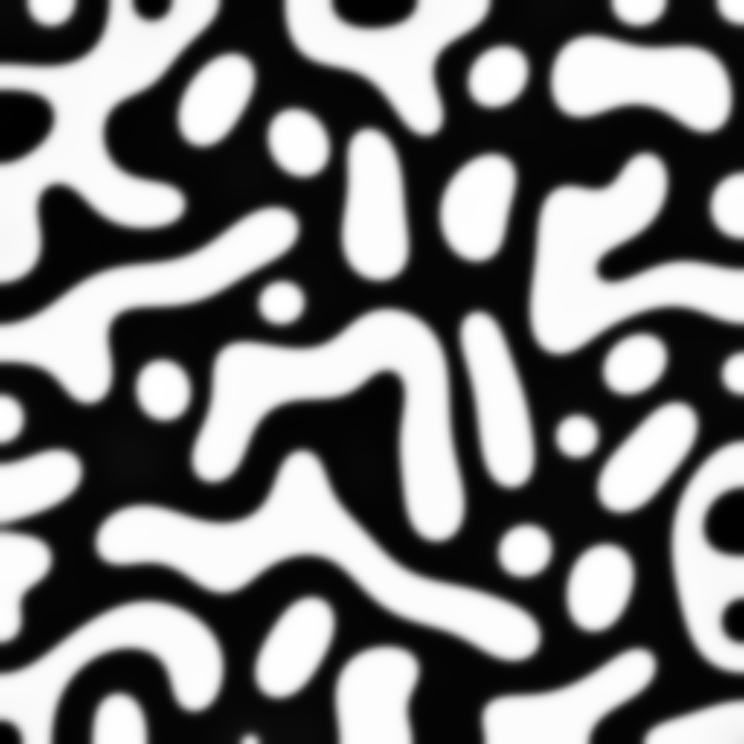}\\
 
 \multicolumn{3}{c}{t=3} \\
 
  \includegraphics[scale=0.14]{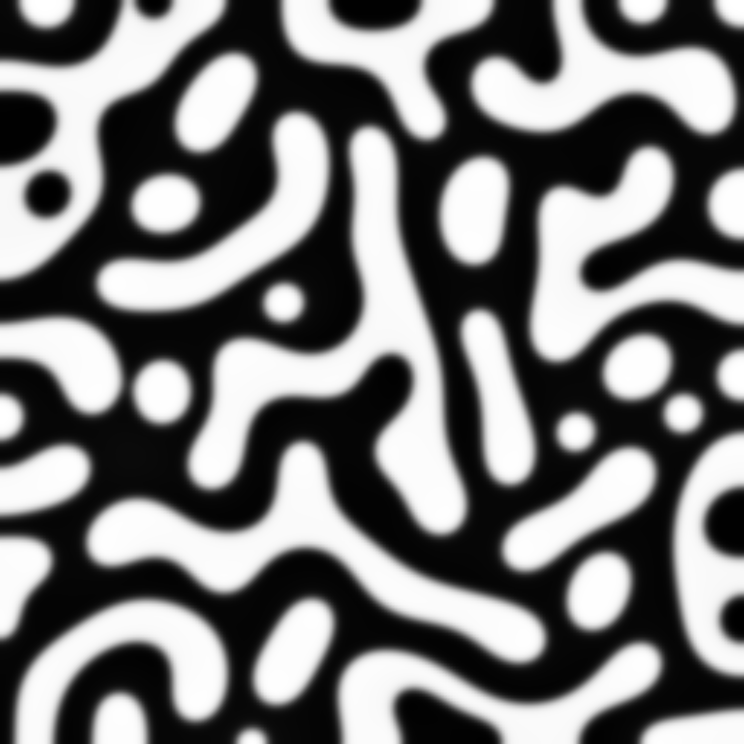}& \includegraphics[scale=0.14]{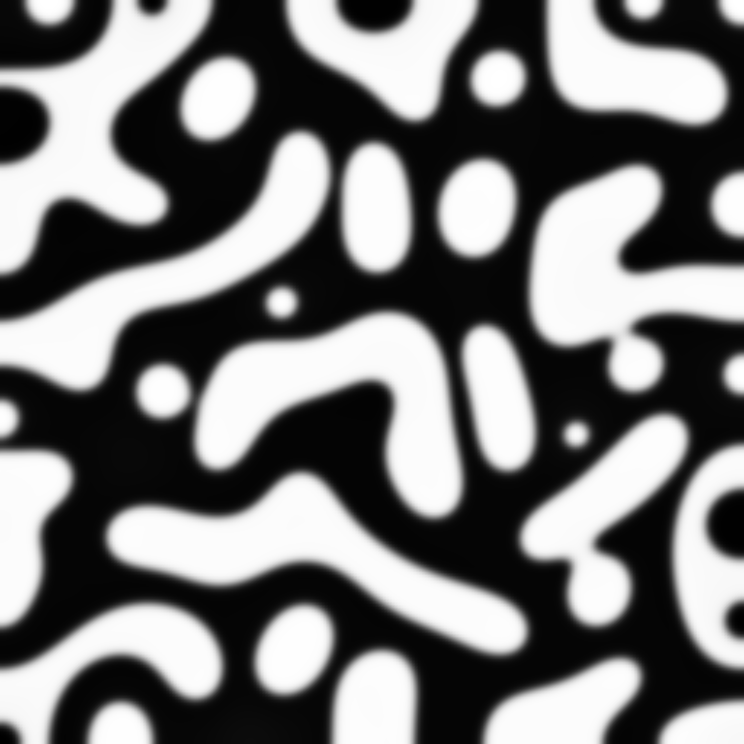} &\includegraphics[scale=0.14]{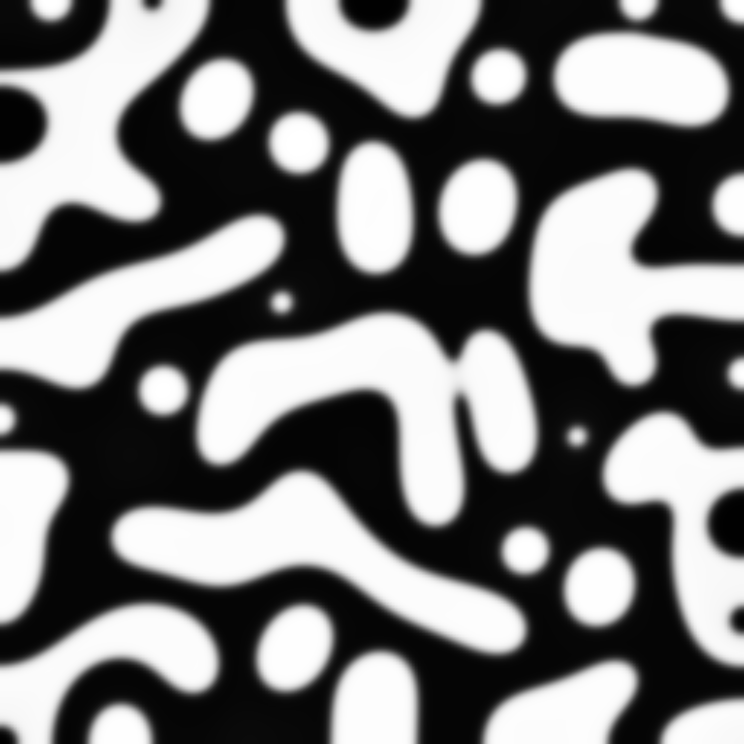}\\

\multicolumn{3}{c}{t=5} \\ 

 $\gamma=0.0$& $\gamma=0.06 $ &$\gamma=0.12$   
 
\end{tabular}
\caption{Snapshots of coarsening of a binary fluid during spinodal decomposition with $\gamma = 0.06$ (second column), $0.12$ (third column), respectively. The case of $\gamma=0$ (first column) is included for comparison purpose. White corresponds to $\phi \approx 1$, and black corresponds to $\phi \approx -1$.  The parameters are $\epsilon=0.03$, $M(\phi)=\sqrt{(1-\phi^2)^2+\epsilon^2}$, $Pe=1$, $\eta=1$ $k=0.05$, $h=\frac{6.4\sqrt{2}}{256}$.}
\label{Spi}
\end{figure}

The parameters for the computations are similar to those in \cite{Wise2010}: $\Omega=[0, 6.4]\times [0, 6.4]$, $\epsilon=0.03$, $h=\frac{6.4\sqrt{2}}{256}$, $k=0.05$, $Pe=1.0$, $\eta=0.083$ and $m(\phi)=\sqrt{(1+\phi)^2(1-\phi)^2+\epsilon^2}$. We use three values $0.0, 0.06, 0.12$ for $\gamma$. The case of $\gamma=0.0$ is included for comparison purpose. For the initial condition of the phase field variable, we take a random field of values $\phi_0=\bar{\phi}+r(x,y)$ with an average composition $\bar{\phi}=-0.05$ and random $r \in [-0.05, 0.05]$. The boundary conditions are given in \eqref{noflux1}-\eqref{noflux2}.
Fig. \ref{Spi} shows the filled contour plot of $\phi$ in gray scale for cases $\gamma=0.0, 0.06, 0.12$. Fig. \ref{EnCom} shows the evolution of the discrete scaled surface energy  \eqref{SGLenergy} in the time interval $[0, 5]$.

The results showing in Figure \ref{Spi} are comparable to those in \cite{Wise2010, GXX2014}. At early stage of spinodal decomposition $(t=1)$, the patterns in the three cases are statistically similar. Later on,  
the systems with larger $\gamma$ tend to straighten their interface faster, and the identified fluid islands are fatter, which indicate a faster coarsening rate. At $t=5$,  the patterns in the second column and third column of Fig. \ref{Spi} reveal the phenomenon of islands merging due to fluid flow, in comparison with the case $\gamma=0$ (first column) where coarsening is mainly realized through surface diffusion. The larger $\gamma$ is, the richer the islands connection is. In addition, one can observe the Ostwald ripening in all three cases: larger droplets grow at the expense of smaller ones.

Coarsening rate can be tied with the energy decay rate \cite{KoOt2002}. Fig. \ref{EnCom} further corroborates the conclusion that larger $\gamma$ leads to faster coarsening rate. The discrete scaled energy matches with each other at the early stage of spinodal decomposition. At later time, the energy is decreasing slightly faster for the cases with larger $\gamma$.

\begin{figure}[h!]
\centering
 \includegraphics[width=0.9\linewidth]{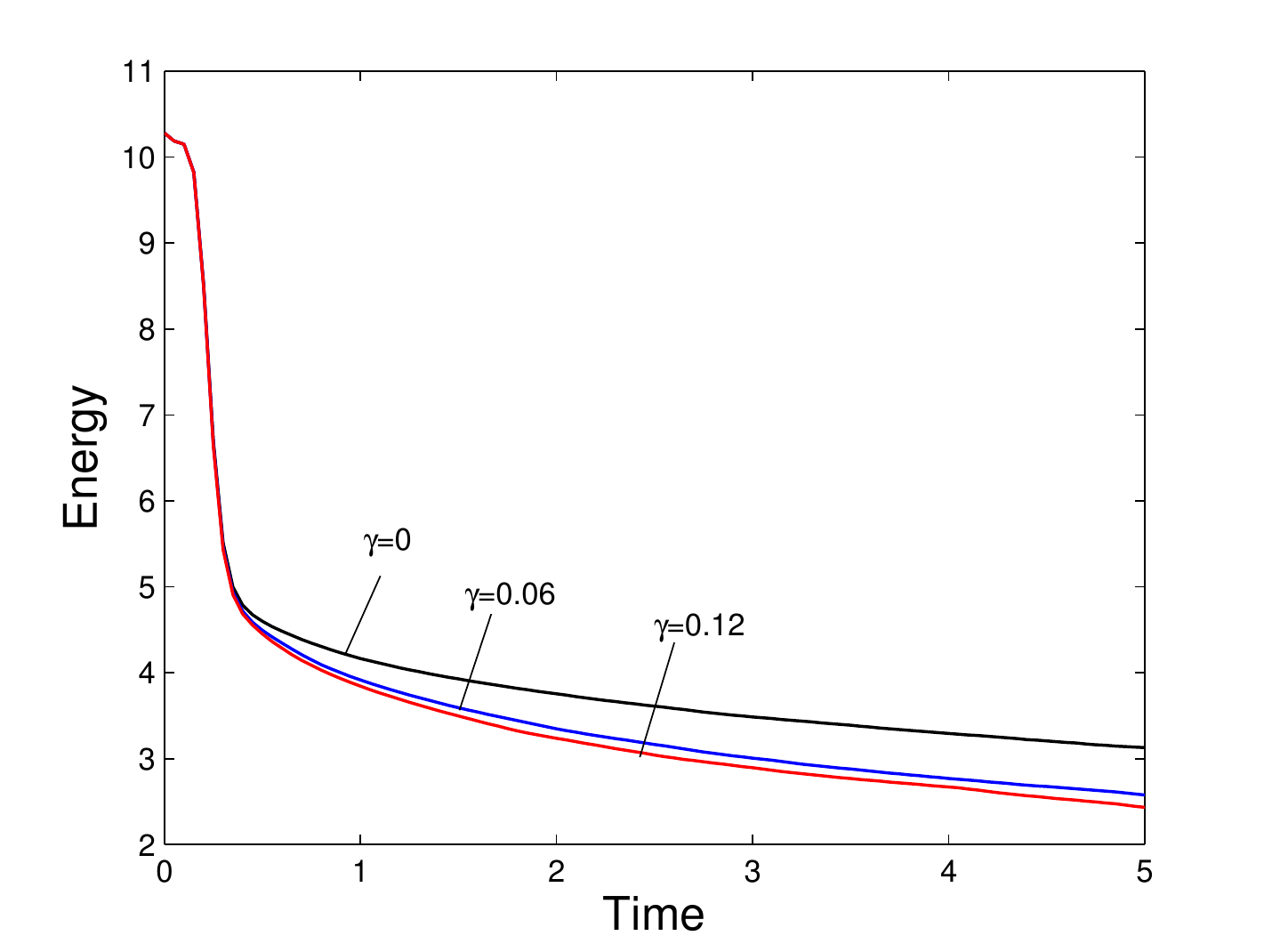}
 \caption{The discrete scaled surface  energy \eqref{SGLenergy} plotted as a function of time for the simulation of spinodal decomposition. The case of $\gamma=0$  is included for comparison purpose.  The rest of the parameters are the same as in Figure \ref{Spi}.  }
  \label{EnCom}
\end{figure}

\subsection{Interface break-up and adaptive mesh refinement}
As shown in \cite{KKL2004, KSW2008, ShYa2010b} among many others, at least 4 to 8 grid cells across the interfacial region are needed for low order methods such as P1 finite element to resolve the interface dynamics accurately. Thus adaptive mesh refinement is indispensable to achieve accuracy and efficiency for low order methods. In this set of numerical experiments, we show that our numerical scheme effected with the adaptive mesh refinement of FreeFem++ \cite{Hecht2012} is capable of capturing the  topological transition of the interface (e.g. interface break-up) smoothly.

The set-up of the experiment is similar to the Rayleigh-Taylor instability. We consider a light fluid layer initially sandwiched by two heavy fluid layers in a square domain $\Omega=[0, 2\pi] \times [0, 2\pi]$. 
For simplicity, we assume that the density variance of two fluids is small so that a Boussinesq approximation can be employed.  
Specifically, we take the background density as $1.0$ and add the following buoyancy term to the Darcy equation in \eqref{CHD} 
$$-b(\phi)\hat{\mathbf{y}}=-G(\rho(\phi)-\bar{\rho})\hat{\mathbf{y}}=-G\frac{\rho_1-\rho_2}{2}(\phi-\bar{\phi})\hat{\mathbf{y}}:= -\lambda(\phi-\bar{\phi})\hat{\mathbf{y}},$$ 
where $\hat{\mathbf{y}}$ is the unit vector pointing upwards ($\hat{\mathbf{y}}=(0,1)$), $G$ is the gravitational constant, $\rho(\phi)=\frac{1+\phi}{2}\rho_1+\frac{1-\phi}{2}\rho_2$ with $\rho_2 \approx \rho_1=1.0$, $\bar{\rho}$ is the spatially averaged density , $\bar{\phi}$ is the spatially averaged order parameter, and $\lambda=G\frac{\rho_1-\rho_2}{2}$. Introducing two flat interfaces with small perturbations
\begin{align*}
&y_1(x)=\pi -(0.5+0.1\cos(x)), \quad y_2(x)=\pi +(0.5+0.1\cos(x)),
\end{align*}
then the initial condition for the phase field variable is defined as (see also \cite{LLG2002b})
\begin{align*}
\phi_0=\tanh \big(\frac{y-y_1(x)}{\sqrt{2}\epsilon}\big)\tanh \big(\frac{y-y_2(x)}{\sqrt{2}\epsilon}\big).
\end{align*}
Fig. \ref{Ini} shows the initial configuration of the phase field variable where $\epsilon=0.01$.
\begin{figure}[h!]
\centering
 \includegraphics[width=0.8\linewidth]{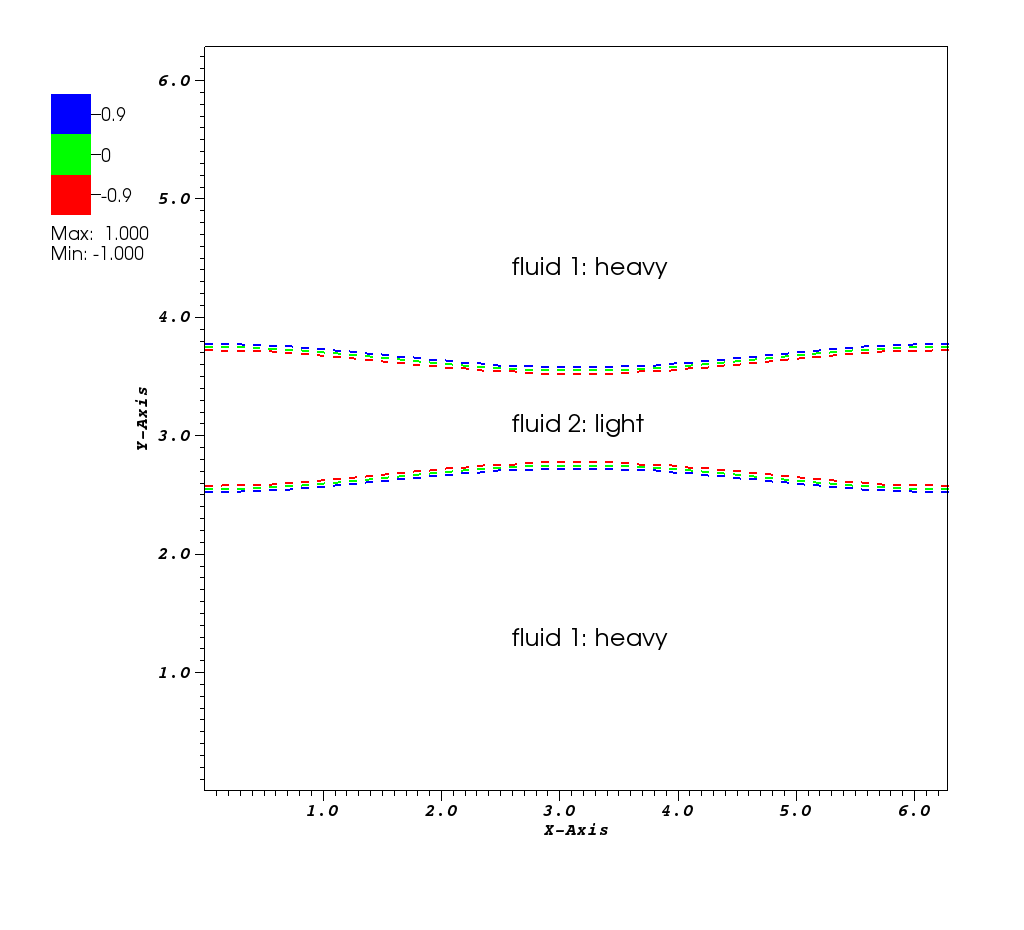}
 \caption{The initial configuration of the phase field variable. $\epsilon=0.01$. Three contours in the upper interfacial layer correspond to $\phi=0.9, 0, -0.9$ from top to bottom.   }
  \label{Ini}
\end{figure}

\begin{figure}[h!]
\centering
\begin{tabular}{cc}
 \includegraphics[scale=0.185]{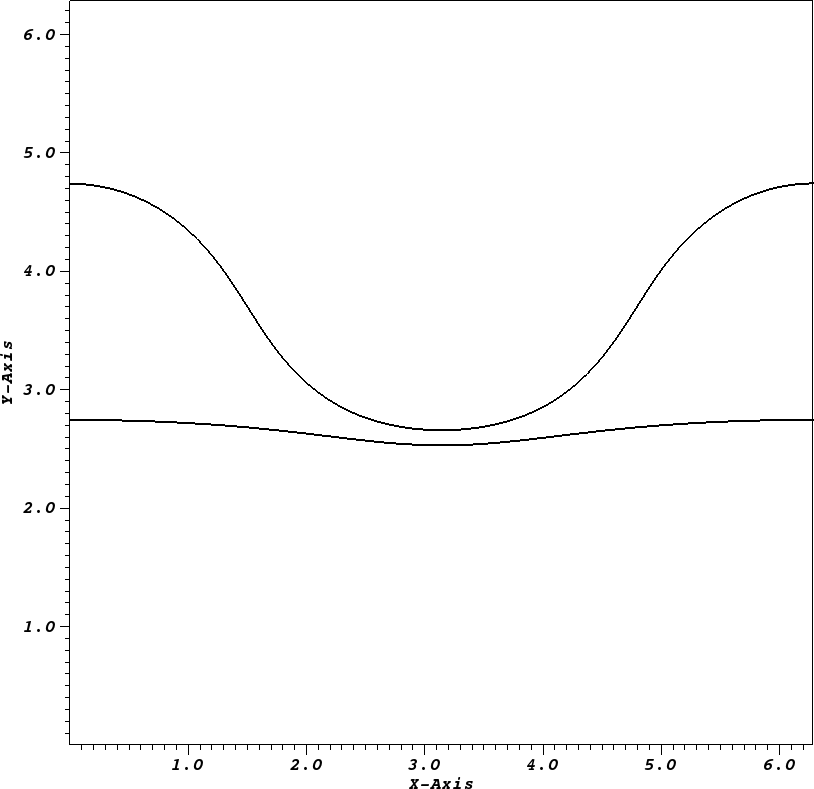} &\includegraphics[scale=0.185]{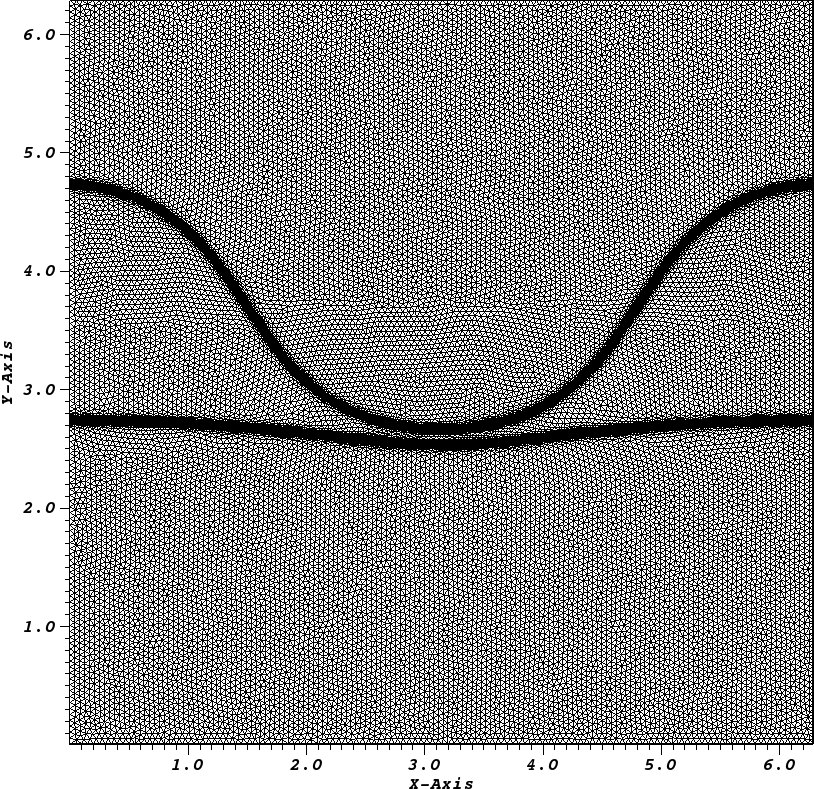} \\
 \multicolumn{2}{c}{t=8.5} \\
 
 \includegraphics[scale=0.185]{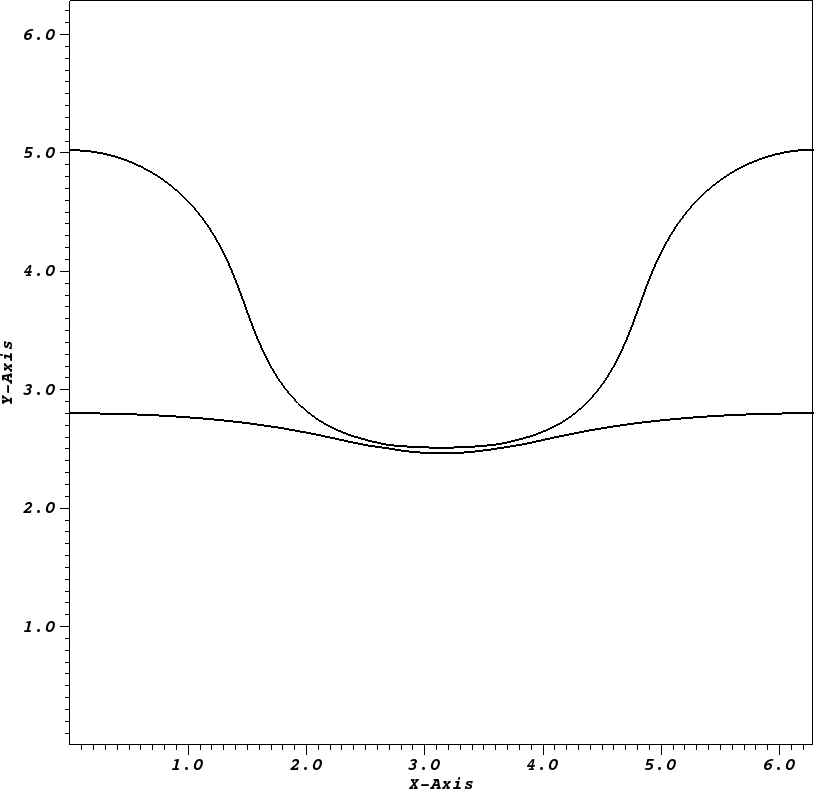} &\includegraphics[scale=0.185]{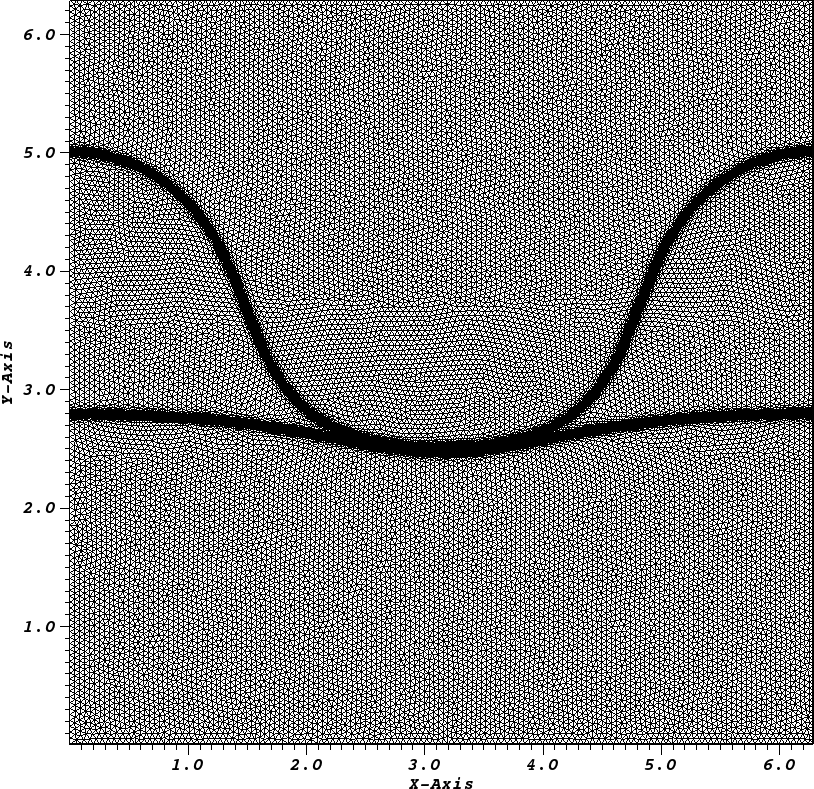}\\
 \multicolumn{2}{c}{t=9.5} \\
 
 \includegraphics[scale=0.185]{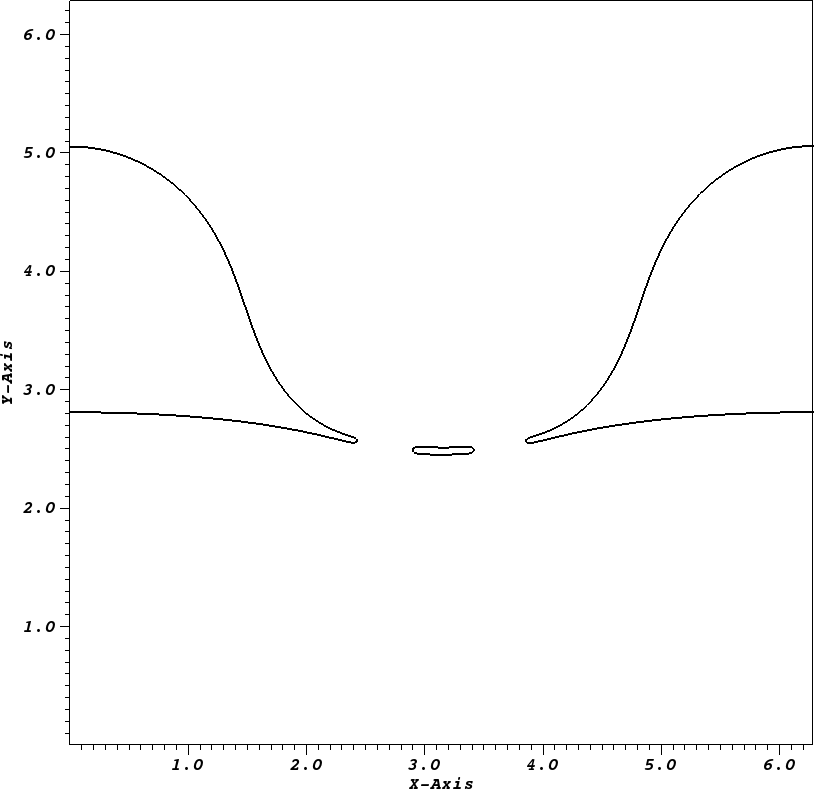} &\includegraphics[scale=0.185]{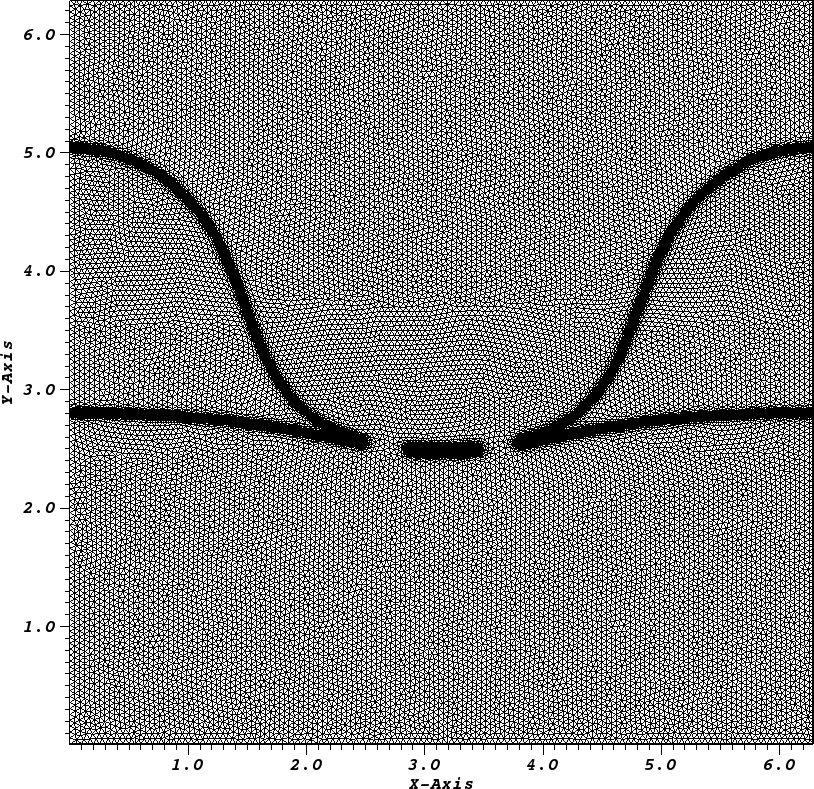}\\
 \multicolumn{2}{c}{t=9.6} \\
 
 \includegraphics[scale=0.185]{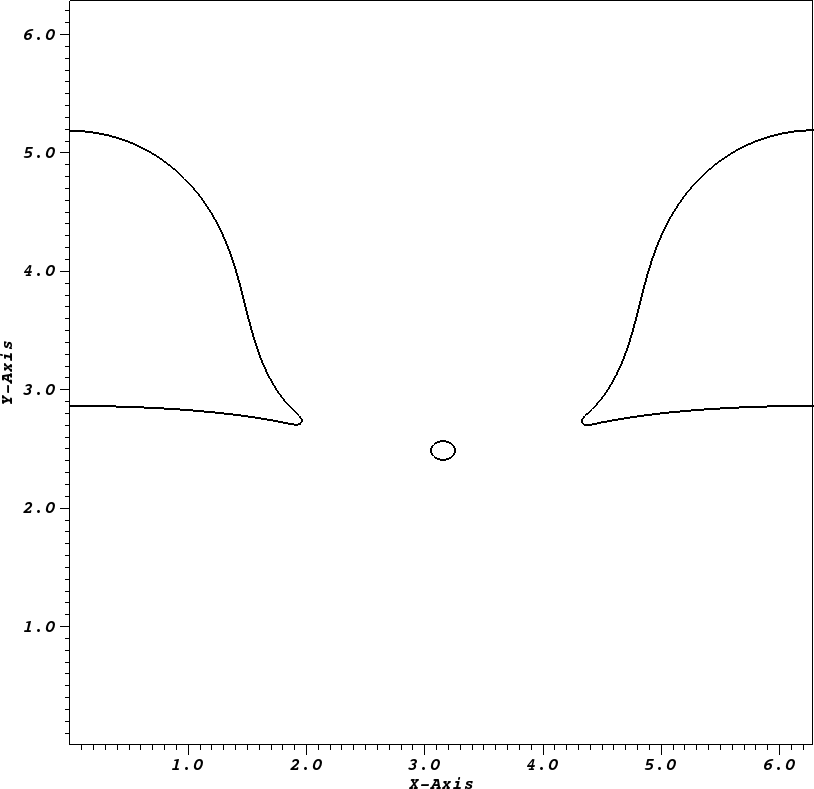} &\includegraphics[scale=0.185]{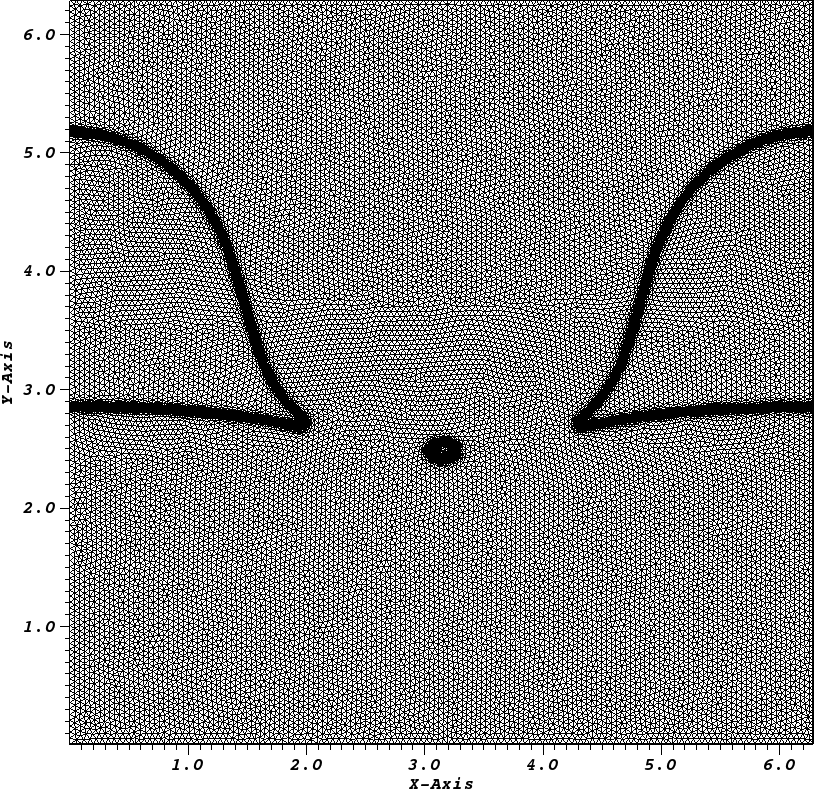}\\
 \multicolumn{2}{c}{t=10}  
 
\end{tabular}
\caption{Snapshots of the zero contour of $\phi$  (left column) and the associated meshes (right column). $\epsilon=0.01$,  $Pe=100$, $\gamma=0.25$, $m(\phi)=1.0$, $\lambda=2.946$, $k=0.005$ and $\eta(\phi)=\frac{1+\phi}{2} \eta_1+\frac{1-\phi}{2}  \eta_2$ with $\eta_1=0.1, \eta_2=0.5$.}
\label{IntBr}
\end{figure}

\newpage
In the following computation, we take $\epsilon=0.01$,  $Pe=100$, $\gamma=0.25$, $m(\phi)=1.0$, $\lambda=2.946$, and $\eta(\phi)=\frac{1+\phi}{2} \eta_1+\frac{1-\phi}{2}  \eta_2$ with $\eta_1=0.1, \eta_2=0.5$. For time stepsize, we choose $k=0.005$. In space, we use the $P_2$ finite element space for all variables, and we adapt the mesh every five time steps according to the Hessian of the order parameter such that at least 2 grid cells are located across the diffuse interface. Homogeneous Neumann boundary conditions are imposed for $\phi$ and $\mu$, and no-penetration boundary condition is prescribed on velocity. Snapshots of the zero contour of the order parameter are shown in Fig. \ref{IntBr}, along with the underlying meshes. In the language of sharp interface models, the upper interface is unstably stratified. The heavy fluid layer penetrates the light fluid layer, eventually causes the light fluid layer breaking up into three drops. The break-up event of the zero contour is captured by our numerical algorithm. The effectiveness of the empirical adaptive mesh refinement can be observed from Fig. \ref{IntBr} as well where one can not differentiate the triangles in the interfacial region due to the dense density there.

\section{Conclusions}

In this paper, we have presented a novel  time discretization scheme for the Cahn-Hilliard-Hele-Shaw system with variable viscosity and mobility that models two-phase flow in a Hele-Shaw cell or porous media. The scheme is very efficient since the update of pressure is completely decoupled from that of phase field variable, and the pressure update involves only a Poisson problem with constant coefficient.
The fully discrete numerical scheme effected with finite-element method is shown to be unconditionally stable. We verify the first order in time convergence of the fully discrete scheme by performing Cauchy convergence test. We also test our scheme on simulating spinodal decomposition of a binary fluid in a Hele-Shaw cell. Our results are comparable to those produced by coupled schemes as reported in \cite{Wise2010, GXX2014}. Finally, we show that our numerical scheme effected with adaptive mesh refinement is able to capture topological transitions of the interface smoothly. To the best of the author's knowledge, this is the first numerical scheme that decouples the pressure and the phase field variables in the numerical simulation of Cahn-Hilliard-Hele-Shaw system while maintaining unconditional stability.

There are several potential extensions of the current work. The decoupling feature in the design is especially attractive for the numerics of Cahn-Hilliard fluid models which typically require nonlinear solvers for stability's concern. The extension of the current scheme to the case of variable density, or to the case of coupled Cahn-Hilliard-Stokes-Darcy system in karstic geometry would be interesting \cite{HSW13, HWW2014}. On the theoretical side, the rigorous error analysis of the scheme, especially with variable mobility and viscosity,  is a challenging topic.  

\section*{Acknowledgments} 
This work was completed while the author was supported as a Research Assistant on an NSF grant
(DMS1312701). The author also acknowledges the support of NSF DMS1008852, a planning
grant and a multidisciplinary support grant from the Florida State University. The author thanks Dr. X. Wang and Dr. S.M. Wise for some insights into the problem and many helpful conversations.

\bibliography{multiphase}
\bibliographystyle{bmc-mathphys.bst}

\end{document}